\documentclass[11pt]{article}
\usepackage{hyperref}
\usepackage{url}
\usepackage{amsmath}
\usepackage{bm}
\usepackage{amsfonts}
\usepackage{algorithm}
\usepackage{algorithmic}
\usepackage{amsthm,paralist}
\usepackage{enumitem,graphicx}
\usepackage{color}

\usepackage[scriptsize,bf]{caption}
\usepackage{amssymb}

\newcommand{\beq}{\begin{equation}}
\newcommand{\eeq}{\end{equation}}
\newcommand{\st}{{\rm s.t.}}

\newcommand{\cO}{{\mbox{$\mathcal{O}$}}}

\newcommand{\tA}{{\mbox{$\widetilde{A}$}}}

\newcommand{\tP}{{\mbox{$\widetilde{P}$}}}

\newtheorem{theorem}{Theorem}[section]
\newtheorem{lemma}{Lemma}[section]
\newtheorem{corollary}{Corollary}[section]

\newtheorem{remark}{Remark}[section]

\def\smskip{\par\vskip 5 pt}

\def\QED{\hfill{\bf Q.E.D.}\smskip}

\title{Improved Iteration Complexity Bounds of Cyclic Block Coordinate Descent for Convex Problems}

\author{{Ruoyu Sun} \thanks{Department of Management Science and Engineering,  Stanford University, Stanford, CA. \texttt{ruoyu@stanford.edu}.}, {Mingyi Hong}\thanks{Department of Industrial and Manufacturing Systems Engineering, Iowa State University, Ames, IA,  \texttt{mingyi@iastate.edu}. } \
 \thanks{ The short version of the paper has been submitted to NIPS 2015 on June 5th, 2015
 and accepted. The authors contribute equally to this work.}
\date{\today} 
}

\begin{document}

\maketitle
\vspace{-0.5cm}
\begin{abstract}
The iteration complexity of the block-coordinate descent (BCD) type algorithm has been under extensive investigation. It was recently shown that for convex problems the classical cyclic BCGD (block coordinate gradient descent) achieves an $\cO(1/r)$ complexity ($r$ is the number of passes of all blocks). However, such bounds are at least linearly depend on $K$ (the number of variable blocks), and are at least $K$ times worse than those of the gradient descent (GD) and proximal gradient (PG) methods.
In this paper, we aim to close such theoretical performance gap between cyclic BCD and GD/PG. First we show that for a family of quadratic nonsmooth problems, the complexity bounds for cyclic Block Coordinate Proximal Gradient (BCPG), a popular variant of BCD, can match those of the GD/PG in terms of dependency on $K$ (up to a $\log^2(K)$ factor).
{\color{black}
For the same family of problems, we also improve the bounds of the classical BCD (with exact block minimization) by an order of $K$.} Second, we establish an improved complexity bound of Coordinate Gradient Descent (CGD)
 for general convex problems which can match that of GD in certain scenarios.
 Our bounds are sharper than the known bounds as they are always at least $K$ times worse than GD.
 {\color{black} Our analyses do not depend on the update order of block variables inside each cycle, thus our results
 also apply to BCD methods with random permutation (random sampling without replacement, another popular variant). }
\end{abstract}

\vspace{-0.3cm}
\section{Introduction}
\vspace{-0.3cm}

Consider the following convex optimization problem
\begin{align}\label{eq:origin}
\begin{split}
\min&\quad f(x)=g(x_1,\cdots, x_K) + \sum_{k=1}^{K}h_k(x_k),\quad \st \; \quad x_k\in X_k, \ \forall~k=1,\cdots K.
\end{split}
\end{align}
where $g: X\to \mathbb{R}$ is a convex smooth function; $h: X\to \mathbb{R}$ is a convex lower semi-continuous possibly nonsmooth function; $x_k\in X_k\subseteq \mathbb{R}^{N}$ is a block variable.
A very popular method for solving this problem is the so-called block coordinate descent (BCD) method \cite{bertsekas99}, where each time a single block variable is optimized while the rest of the variables remain fixed. Using the classical cyclic block selection rule, the BCD method is described in the following table.
\vspace{-0.3cm}
\begin{center}
\fbox{
\begin{minipage}{5.2in}
\smallskip
\centerline{\bf Algorithm 1: The Cyclic Block Coordinate Descent (BCD)}
\smallskip
At each iteration $r+1$, update the variable blocks by:
\begin{equation}\label{eq:BCM}
x^{(r)}_k \in \min_{x_k\in X_k}\; g\left(x_k, w^{(r)}_{-k}\right)+h_k(x_k), \; k=1,\cdots, K.\\
\end{equation}
\end{minipage}
}
\end{center}
\!where we have used the following short-handed notations:{
\begin{align}
w^{(r)}_k&:=\left[x^{(r)}_1,\cdots, x^{(r)}_{k-1}, x^{(r-1)}_{k}, x^{(r-1)}_{k+1}, \cdots, x^{(r-1)}_K\right], \ k=1,\cdots, K,\nonumber\\
w^{(r)}_{-k}&:=\left[x^{(r)}_1,\cdots, x^{(r)}_{k-1}, x^{(r-1)}_{k+1}, \cdots, x^{(r-1)}_K\right], \ k=1,\cdots, K,\nonumber\\
x_{-k}&:=\left[x_1,\cdots, x_{k-1}, x_{k+1}, \cdots, x_K\right]. \nonumber
\end{align}}
\!The convergence analysis of the BCD has been extensively studied in the literature, see \cite{bertsekas99, PowellBCD73,tseng01,Razaviyayn12SUM,Beck13,hong13complexity,Grippo00,Luo92CD,xu12}.
 {\color{black}
For example it is known that for smooth problems (i.e. $f$ is continuous differentiable but possibly nonconvex, $h=0$),
 if each subproblem has a unique solution and $g$ is non-decreasing in the interval between the current iterate and the minimizer of the subproblem
 (one special case is per-block strict convexity), then every limit point of $\{x^{(r)}\}$ is a stationary point \cite[Proposition 2.7.1]{bertsekas99}.
The authors of \cite{Grippo00,tseng01} have derived relaxed conditions on the convergence of BCD.  } 
{\color{black}In particular, when problem \eqref{eq:origin} is convex and the level sets are compact, the convergence of the BCD is guaranteed without requiring the subproblems to have unique solutions \cite{Grippo00}. }
Recently Razaviyayn {\it et al} \cite{Razaviyayn12SUM} have shown that the BCD converges if each subproblem \eqref{eq:BCM} is solved inexactly, by way of optimizing certain surrogate functions. 

  Luo and Tseng in \cite{Luo92CD} have shown that when problem  \eqref{eq:origin} satisfies certain additional assumptions such as having a smooth composite objective and a polyhedral feasible set, then BCD converges linearly without requiring the objective to be strongly convex. 
There are many recent works on showing iteration complexity for randomized BCGD (block coordinate gradient descent), see \cite{richtarik12, nestrov12,lu13complexity, meisam14nips, lu13randomized} and the references therein. However the results on the classical cyclic BCD is rather scant. Saha and Tewari \cite{Saha10} show that the cyclic BCD achieves sublinear convergence for a family of special LASSO problems. {Nutini {\it et al} \cite{schmidt15icml} show that when the problem is strongly convex, unconstrained and smooth, BCGD with certain Gauss-Southwell block selection rule could be faster than the randomized rule.} Recently Beck and Tetruashvili show that cyclic BCGD converges sublinearly if the objective is smooth. Subsequently Hong {\it et al} in \cite{hong13complexity} show that such sublinear rate not only can be extended to problems with nonsmooth objective, but is true for a large family of BCD-type algorithm (with or without per-block exact minimization, which includes BCGD as a special case). When each block is minimized exactly and when there is no per-block strong convexity, Beck \cite{Beck13b} proves the sublinear convergence for certain 2-block convex problem (with only one block having Lipschitzian gradient). {It is worth mentioning that all the above results on cyclic BCD directly apply to {\it randomly permuted} BCD in which the blocks are randomly sampled without replacement in each cycle, since the proof techniques do not require the same order to be used in each cycle.
However, we suspect that special proof techniques tailored for randomly permutation
are necessary for establishing tight bounds of randomly permuted BCD; in fact, a recent work of Sun {\it et al} \cite{sun2015expected} on randomly permuted ADMM
provides some theoretical evidence that the cyclic rule is worse than the random permutation rule.
}

To illustrate the rates developed for the cyclic BCD algorithm, let us define $X^*$ to be the optimal solution set for problem \eqref{eq:origin}, and define the constant
\begin{align}\label{eq:r0}
R_0:=\max_{x\in X}\max_{x^*\in X^*}\left\{\|x-x^*\|\mid f(x)\le f(x^{(0)})\right\}.
\end{align}
Let us assume that $h_k(x_k)\equiv 0, \; X_k=\mathbb{R}^N,\; \forall~k$ for now, and assume that $g(\cdot)$ has Lipschitz continuous gradient:
\begin{align}\label{eq:lip}
\|\nabla g(x)-\nabla g(z) \|\le L\|x-z\|, \; \forall~x, z\in X.
\end{align}
Also assume that $g(\cdot, x_{-k})$ has Lipschitz continuous gradient with respect to each $x_k$, i.e.,
\begin{align}\label{eq:lip:block}
\|\nabla_k g(x_k, x_{-k})-\nabla_k g(v_k, x_{-k}) \|\le L_k\|x_k-v_k\|, \; \forall~x, v\in X, \; \forall~k.
\end{align}
Let $L_{\max}:=\max_{k}L_k$ and $L_{\min}:=\min_{k}L_k$. It is known that the cyclic BCPG has the following iteration complexity \cite{Beck13,hong13complexity} \footnote{\footnotesize  Note that the assumptions made in \cite{Beck13} and \cite{hong13complexity} are slightly different, but the rates derived in both cases have similar dependency on the problem dimension $K$. }
\begin{align}
\Delta^{(r)}_{\rm BCD}:=f(x^{(r)})-f^*\le C L_{\max}(1+K L^2/L^2_{\min})R^2_0\frac{1}{r}, \quad \forall~r\ge 1, \label{eq:complexity:bcd}
\end{align}
where $C>0$ is some constant independent of problem dimension.  Similar bounds are provided for cyclic BCD in \cite[Theorem 6.1]{hong13complexity}. In contrast, it is well known that when applying the classical gradient descent (GD) method to problem \eqref{eq:origin} with the constant stepsize $1/L$, we have the following rate estimate \cite[Corollary 2.1.2]{Nesterov04}
\begin{align}\label{eq:complexity:gd}
\Delta^{(r)}_{\rm GD}:=f(x^{(r)})-f(x^*)\le \frac{2\|x^{(0)}-x^*\|^2 L}{r+4}\le \frac{2R^2_0 L}{r+4},\quad \forall \; r\ge 1, \; \forall~x^*\in X^*.
\end{align}
Note that unlike \eqref{eq:complexity:bcd}, here the constant in front of the $1/(r+4)$ term is {\it independent} of the problem dimension. In fact, the ratio of the bound given in \eqref{eq:complexity:bcd} and \eqref{eq:complexity:gd} is
$$\frac{C L_{\max}}{L}(1+K L^2/L^2_{\min})\frac{r+4}{r}$$
which is at least in the order of $K$. For big data related problems with over millions of variables, a multiplicative constant in the order of $K$ can be a serious issue. In a recent work by Saha and Tewari \cite{Saha10}, the authors show that for a LASSO problem with special data matrix, the rate of cyclic BCD (with special initialization) is indeed $K$-independent. Unfortunately, such a result has not yet been extended to any other convex problems. An open question posed by a few authors \cite{Beck13, Beck15, Saha10} are: is such a $K$ factor gap intrinsic to the cyclic BCD or merely an artifact of the existing analysis?

\subsection{Summary of Contributions}

In this paper, we provide improved iteration complexity bounds of cyclic block coordinate descent methods for
 convex composite function minimization.
 Our analyses do not depend on the update order of block variables inside each cycle, thus our results
 also apply to BCD methods with random permutation (random sampling without replacement, another popular variant).
 Recall that $K$ is the number of blocks,
 $L$ is the global Lipshitz constant of $\nabla g$, and $L_k$ is the Lipshitz constant of $\nabla g$ with respect to the $k$-th block.

 \begin{itemize}
   \item
     For minimizing the sum of a convex quadratic function (not necessarily strongly convex) and sepearable non-smooth function (including LASSO as a special case), we prove that the iteration complexity bound of cyclic block coordinate gradient descent (C-BCGD) with a small stepsize $1/L$
    matches that of GD (gradient descent) up to an $O(\log^2(K))$-factor.
    When the stepsize for the $k$-th block update is $1/L_k$, we establish an iteration complexity bound that is
    $K$-times better than the existing bound, but still $K$ times worse than GD.
    We also improve the bound of the exact BCD by an order of $K$ (note that if each block has size $1$, exact CD is equivalent
    to CGD with stepsize $1/L_k$; otherwise, exact BCD is different from BCGD).

    \item For general smooth convex optimization (i.e. $h_k = 0, \forall k$), we prove a meta iteration complexity bound of cyclic
    coordinate gradient descent (C-CGD) that is proportional to the spectral norm of a ``moving-iterate Hessian''.
    By an preliminary estimate of this spectral norm, this meta bound implies an  iteration complexity bound that, in certain scenarios, matches the bound
    of GD  and $K$-times better than the known bounds.
    For the quadratic case, the moving-iterate Hessian becomes a constant matrix and the meta bound reduces to
     the bound mentioned in the first bullet which matches GD up to an $O(\log^2{2K})$-factor for small stepsize $1/L$.

 \end{itemize}

   For illustration, we summarize the comparison of our main results with other existing bounds in several simple settings in Table \ref{compare results}.
   For simplicity, we only list the results of BCGD for the smooth case; note that the results listed in the last row for QP
   also apply to LASSO (or more general, sum of a quadratic function plus separable non-smooth function),
   and for cyclic exact BCD the results are similar to the last row (see Theorem \ref{thm:bound:bcd}).
    We assume $L_k = L_1, \forall k$ and consider two extreme values of $L/L_k$: $L/L_1 = 1$ and $L/L_1 = K$.
    The first case represents a separable objective function $g(x) = \frac{L}{2} \sum_{i=1}^K \| x_i \|^2  $ with a block diagonal Hessian matrix,
    and the second case represents a highly non-separable objection function $g(x) = \frac{L }{2K} (\sum_{i=1}^K  x_i )^2 $ with a full Hessian matrix, assuming
    all $x_i$'s have the same size.

\begin{table}[!htbp]
\caption{ Comparison of Various Iteration Complexity Results }
\label{compare results}
\vspace{-0.2cm}
\begin{tabular}{|c|c|c|c|}
\hline
Lip-constant & Diag Hessian $ L_i  = L $  & Full Hessian   $L_i  = L/K  $   &  Full Hessian $L_i = L/K $                       \\
1/Stepsize     & $ 1/L $ &  Large stepsize $ K/L $  &   Small stepsize $ 1/L $   \\
\hline
\hline
 GD         & ${L}/{r}$         & N/A                            & $ {L}/{r}  $                      \\
\hline
 {Randomized BCGD }        & ${L}/{r}$         & ${L}/{(Kr)}$                            & $ {L}/{r}  $                      \\
\hline
 Cyclic BCGD \cite{Beck13}    & $K {L}/{r}$        &  $K^2{L}/{r}  $               & $ K {L}/{r}$                     \\
\hline
 Cyclic CGD, Coro \ref{coro: CGD ite complexity} & $K {L}/{r}$      &   $ {{K L}/{r} } $       &  {$ {L}/{r} $  }                      \\
\hline
 Cyclic BCGD (QP), Thm \ref{thm:bound:bcpg}  &  $ \log^2(2K) {L}/{r} $         &   $  \log^2(2K) K{L}/{r}  $       &   $ \log^2(2K) {L}/{r} $                       \\
\hline
\end{tabular}
\end{table}

\section{Improved Bounds of Cyclic BCPG for Nonsmooth Quadratic Problem}\label{sec:BCPG}
In this section, we consider the following nonsmooth quadratic problem
\begin{align}\label{eq:qp:ns}
\min\; f(x):= \frac{1}{2}\left\|\sum_{k=1}^{K}A_k x_k - b\right\|^2+ \sum_{k=1}^{K}h_k(x_k), \quad \st\; x_k\in X_k, \; \forall~k
\end{align}
where $A_k\in\mathbb{R}^{M\times N}$; $b\in\mathbb{R}^{M}$;  $x_k\in\mathbb{R}^{N}$ is the $k$th block coordinate; $h_k(\cdot)$ is the same as in \eqref{eq:origin}. Note the blocks are assumed to have equal dimension for simplicity of presentation. Define $A:=[A_1,\cdots, A_k]\in\mathbb{R}^{M \times KN}$. For simplicity, we have assumed that all the blocks have the same size. Problem \eqref{eq:qp:ns} includes for example LASSO and group LASSO as special cases. 

We consider the following cyclic BCPG algorithm.
\vspace{-0.5cm}
\begin{center}
\fbox{
\begin{minipage}{6.5in}
\smallskip
\centerline{\bf Algorithm 2: The Cyclic Block Coordinate Proximal Gradient (BCPG)}
\smallskip
At each iteration $r+1$, update the variable blocks by:{
\begin{align}\label{eq:BCPG}
x^{(r+1)}_k &= \arg\min_{x_k\in X_k}\; g(w^{(r+1)}_k)+\left\langle \nabla_k g\left(w^{(r+1)}_{k}\right), x_k-x^{(r)}_k\right\rangle+\frac{P_k}{2}\left\|x_k-x^{(r)}_k\right\|^2 +h_k(x_k)
\end{align}}
\end{minipage}
}
\end{center}
Here $P_k$ is the inverse of the stepsize for $x_k$, which satisfies
\begin{align}
P_k \ge \lambda_{\max}\left( A^T_k A_k \right)=L_k, \; \forall~k. \label{eq:Lk}
\end{align}
Define $P_{\max}:=\max_{k}P_k$ and $P_{\min}=\min_k P_k$.
Note that for the least square problem (smooth quadratic minimization, i.e. $h_k \equiv 0, \forall \; k$), BCPG reduces to the widely used BCGD method.

The optimality condition for the $k$th subproblem is given by
\begin{align}\label{eq:opt:bcpg}
\left\langle  \nabla_k g(w^{(r+1)}_k) + P_k (x_k^{(r+1)}-x_k^{(r)}), x_k - x^{(r+1)}_k  \right\rangle +h_k(x_k) - h_k(x^{(r+1)}_k)\ge 0, \; \forall~x_k\in X_k.
\end{align}
In what follows we show that the cyclic BCPG for problem \eqref{eq:qp:ns} achieves a complexity bound that only dependents on $\log^2(K)$, and apart from such log factor it is at least $K$ times better than those known in the literature. Our analysis consists of the following three main steps:
\begin{enumerate}
\item Estimate the  descent of the objective after each BCPG iteration;
\item Estimate the cost yet to be minimized (cost-to-go) after each BCPG iteration;
\item Combine the above two estimates to obtain the final bound.
\end{enumerate}

First we show that the BCPG achieves the sufficient descent.
\begin{lemma}\label{lemma:descent:bcpg}
 We have the following estimate of the descent when using the BCPG:
\begin{align} \label{eq:descent:bcpg}
f(x^{(r)})-f(x^{(r+1)})\ge  \sum_{k=1}^{K}\frac{P_k}{2}\|x^{(r+1)}_k-x^{(r)}_k\|^2.
\end{align}
\end{lemma}
\begin{proof}
We have the following series of inequalities
\begin{align}
&f(x^{(r)})-f(x^{(r+1)}) \nonumber\\
& = \sum_{k=1}^{K} f(w^{(r+1)}_k)-f(w^{(r+1)}_{k+1})\nonumber\\
&\ge \sum_{k=1}^{K}f(w^{(r+1)}_k)-\left(g(w^{(r+1)}_k)+h_k(x^{(r+1)}_k)+\left\langle \nabla_k g\left(w^{(r+1)}_{k}\right), x^{(r+1)}_k-x^{(r)}_k\right\rangle+\frac{P_k}{2}\left\|x^{(r+1)}_k-x^{(r)}_k\right\|^2\right)\nonumber\\
&= \sum_{k=1}^{K}h_k(x_k^{(r)})-h_k(x_k^{(r+1)})-\left(\left\langle \nabla_k g\left(w^{(r+1)}_{k}\right), x^{(r+1)}_k-x^{(r)}_k\right\rangle+\frac{P_k}{2}\left\|x^{(r+1)}_k-x^{(r)}_k\right\|^2\right)\nonumber\\
&\ge \sum_{k=1}^{K}\frac{P_k}{2}\|x^{(r+1)}_k-x^{(r)}_k\|^2 \nonumber.
\end{align}
where the  second inequality uses the optimality condition \eqref{eq:opt:bcpg}. \QED
\end{proof}


To proceed, let us introduce two matrices $\tP$ and $\tA$ given below, which have dimension $K\times K$ and $MK\times NK$, respectively{\small
\begin{align}
\tP&:=\mbox{blkdg}\{P_k\}_{k=1}^{K}=\left[\begin{array}{llllll}
P_1 & 0 &  0& \cdots& 0 & 0\\
0 & P_2& 0& \cdots& 0 &0\\
\vdots&\vdots&\vdots&\cdots&\vdots&\vdots \\
0& 0& 0& \cdots& 0& P_K
\end{array}\right],\; \tA:=\mbox{blkdg}\{A_k\}_{k=1}^{K}=\left[\begin{array}{llllll}
A_1 & 0 &  0& \cdots& 0 & 0\\
0 & A_2& 0& \cdots& 0 &0\\
\vdots&\vdots&\vdots&\cdots&\vdots&\vdots \\
0& 0& 0& \cdots& 0& A_K
\end{array}\right].\nonumber
\end{align}}
By utilizing the definition of $P_k$ in \eqref{eq:Lk} we have the following inequalities (the second inequality comes from \cite[Lemma 1]{nestrov12}){
\begin{align}\label{eq:A:L}
\tP \otimes I_N\succeq \tA^T\tA, \quad K\tA^T\tA \succeq A^T A
\end{align}}
\!\!where $I_N$ is the $N\times N$ identity matrix and the notation ``$\otimes$" denotes the Kronecker product.

Next let us estimate the cost-to-go. 
\begin{lemma}\label{lemma:cost:bcpg}
 We have the following estimate of the optimality gap when using the BCPG:{
\begin{align}\label{eq:cost:bcpg}
\Delta^{(r+1)}:&=f(x^{(r+1)})-f(x^*)\nonumber\\
&\le R_0{\color{black}\log(2NK)} \left(L/\sqrt{P_{\min}}+\sqrt{P_{\max}}\right)\left\|(x^{(r+1)}-x^{(r)})(\tP^{1/2}\otimes I_N)\right\|
\end{align}}
\end{lemma}
\begin{proof}
First note that $g(x)$ being convex quadratic implies that its second order Taylor expansion is tight
\begin{align}
g(x^{*})-g(x^{(r+1)})=\langle\nabla g(x^{(r+1)}), x^* - x^{(r+1)}\rangle+\frac{1}{2}\|A (x^{(r+1)}-x^*)\|^2.
\end{align}
Using this fact we can estimate $f(x^{(r+1)})-f(x^*)$ by the following series of inequalities
\begin{align}
&f(x^{(r+1)})-f(x^*)\nonumber\\
& = \left\langle \nabla g(x^{(r+1)}), x^{(r+1)}-x^* \right\rangle + h(x^{(r+1)}) - h(x^*)-\frac{1}{2}\|A (x^{(r+1)}-x^*)\|^2 \nonumber\\
 &\stackrel{\rm (i)}\le\sum_{k=1}^{K}\left\langle \nabla_k g(x^{(r+1)}) -\nabla_k g(w^{(r+1)}_{k}), x^{(r+1)}_k-x^*_k  \right\rangle- \sum_{k=1}^{K}P_k\left\langle x^{(r+1)}_k-x^{(r)}_k,x^{(r+1)}_k-x^*_k\right\rangle-\frac{1}{2}\|A (x^{(r+1)}-x^*)\|^2 \nonumber\\
 &= \sum_{k=1}^{K}\left\langle \left(\sum_{j\ge k}A_j(x^{(r+1)}_j-x^{(r)}_j)\right), A_k(x^{(r+1)}_k-x^*_k)  \right\rangle-(x^{(r+1)}-x^{(r)})^T (\tP\otimes I_N) (x^{(r+1)}-x^{*})\nonumber\\
 &\quad -\frac{1}{2}\|A (x^{(r+1)}-x^*)\|^2 \nonumber\\
    &= (x^{(r+1)}-x^{(r)})^T\tA^T (D_1\otimes I_M) \tA (x^{(r+1)}-x^*)-(x^{(r+1)}-x^{(r)})^T (\tP\otimes I_N) (x^{(r+1)}-x^{*})\nonumber\\
    &\quad-\frac{1}{2}\|A (x^{(r+1)}-x^*)\|^2 \nonumber\\
    &\le (x^{(r+1)}-x^{(r)})^T\left(\tA^T (D_1\otimes I_M) \tA -\tP\otimes I_N\right)(x^{(r+1)}-x^*)\label{eq:cost-to-go1:bcpg}
\end{align}
where in $\rm (i)$ we have used the optimality condition of the subproblem \eqref{eq:BCPG} (i.e., \eqref{eq:opt:bcpg}); in the last equality we have defined a lower triangular matrix $D_1$
\begin{align}\label{eq:D1}
D_1:= \left[\begin{array}{llllll}
1 & 0 &  0& \cdots& 0 & 0\\
1 & 1& 0& \cdots& 0 &0\\
1& 1& 1& \cdots& 0&0\\
\vdots&\vdots&\vdots&\cdots&\vdots&\vdots \\
1& 1& 1& \cdots& 1& 1
\end{array}\right]\in\mathbb{R}^{K\times K}.
\end{align}
{\color{black} Notice that the following is true
\begin{align}
  \tA^T (D_1\otimes I_M) \tA = \left(A^T A - \tA^T \tA\right) \odot D_2 +\tA^T \tA \nonumber,
\end{align}
where ``$\odot$" denotes the Hadamard product; $D_2$ is a lower triangular matrix similarly as defined in \eqref{eq:D1}, but of dimension $KN\times KN$.} Combining this identity and \eqref{eq:cost-to-go1:bcpg}, we have
\begin{align}
  \Delta^{(r+1)}&\le \left((x^{(r+1)}-x^{(r)})(\tP^{1/2}\otimes I_N)\right)^T \bigg((\tP^{-1/2}\otimes I_N)(A^T A -\tA^T\tA) \odot D_2 \nonumber\\
  &\quad +(\tP^{-1/2}\otimes I_N)\tA^T\tA-\tP^{1/2}\otimes I_N\bigg)(x^{(r+1)}-x^*) \nonumber\\
  &\stackrel{\rm (i)}\le \left\|(x^{(r+1)}-x^{(r)})(\tP^{1/2}\otimes I_N)\right\| \left\|(\tP^{-1/2}\otimes I_N)(A^T A -\tA^T\tA)\odot D_2 \right\|\|x^{(r+1)}-x^*\|\nonumber\\
  &\stackrel{\rm (ii)}\le \left\|(x^{(r+1)}-x^{(r)})(\tP^{1/2}\otimes I_N)\right\|  \left\|(\tP^{-1/2}\otimes I_N)(A^T A-\tA^T\tA) \right\|\left(1+\frac{1}{\pi}+\frac{\log(NK)}{\pi}\right)R_0\nonumber\\
  &\stackrel{\rm (iii)}\le R_0\left\|(x^{(r+1)}-x^{(r)})(\tP^{1/2}\otimes I_N)\right\|  \left\|(\tP^{-1/2}\otimes I_N)(A^T A-\tA^T\tA)\right\|\log(2NK)\nonumber\\
  &\le R_0\log(2NK) \left(L/\sqrt{P_{\min}}+\sqrt{P_{\max}}\right)\left\|(x^{(r+1)}-x^{(r)})(\tP^{1/2}\otimes I_N)\right\|
\end{align}
 where $\rm(i)$ uses the Cauchy-Schwartz inequality and the fact that  $\tP\otimes I_N\succeq \tA^T\tA$; $\rm(iii)$ is true for all $KN\ge 3$. Inequality $\rm(ii)$ is true due to a result on the spectral norm of the triangular truncation operator; see \cite[Theorem 1]{ANGELOS92}. In particular, Define
$$Y_{KN}(D_2)= \max\left\{\frac{\|Z\odot D_2\|}{\|Z\|}: Z\in\mathbb{R}^{KN\times KN}, Z\ne 0\right\}.$$
Then we have the following estimate
$$\left|\frac{Y_{KN}(D_2)}{\log (KN)}-\frac{1}{\pi}\right|\le\left(1+\frac{1}{\pi}\right)\frac{1}{\log(KN)}.$$
The proof is completed.
\QED
\end{proof}


Our third step combines the previous two steps and characterizes the iteration complexity. 
This is the main result of this section.
{\color{black}\begin{theorem}\label{thm:bound:bcpg}
The iteration complexity of using BCPG to solve \eqref{eq:qp:ns} is given below.
\begin{enumerate}
\item  When the stepsizes are chosen conservatively as $P_k=L, \; \forall~ k$, we have{
\begin{align}\label{eq:bcpg:bound:kstep}
\Delta^{(r+1)}&\le 3\max\left\{\Delta^0  , 4\log^2(2NK) L\right\}\frac{R^2_0}{r+1}
\end{align}}
\vspace{-0.4cm}
\item When the stepsizes are chosen as  $ P_k=\lambda_{\max}(A_k^T A_k)=L_k,\ \forall~ k $.  
%
Then we have{
\begin{align}\label{eq:bcpg:bound}
\Delta^{(r+1)}& 
  \le 3\max\left\{\Delta^0  , 2\log^2(2NK)\left(L_{\max}+ \frac{L^2}{L_{\min}}\right)\right\}\frac{R^2_0}{r+1}
\end{align}}
\!In particular, if the problem is smooth and unconstrained, i.e., when $h\equiv 0$, and $X_k = \mathbb{R}^{N}, \forall~k$, then we have{
\begin{align}\label{eq:bcpg:bound:uncons}
\Delta^{(r+1)} \le3\max\left\{L, 2\log^2(2NK) \left(L_{\max}+\frac{L^2}{L_{\min}}\right)\right\}\frac{R^2_0}{r+1}. 
\end{align}}
\end{enumerate}
\end{theorem}}
\begin{proof}
For notational simplicity, let us define
$$C:=R_0\log(2NK)\left(L/\sqrt{P_{\min}}+\sqrt{P_{\max}}\right).$$
Taking a square of the cost-to-go estimate \eqref{eq:cost:bcpg} and the sufficient descent estimate \eqref{eq:descent:bcpg}, we obtain
\begin{align*}
(\Delta^{(r+1)})^2&\le C^2\left\|(x^{(r+1)}-x^{(r)})(\tP^{1/2}\otimes I_N)\right\|^2\nonumber\\
&= C^2\sum_{k=1}^{K}P_k\|x_k^{(r+1)}-x_k^{(r)}\|^2\nonumber\\
&\le 2{C^2}(\Delta^{(r)}-\Delta^{(r+1)}).
\end{align*}
Utilizing a result from \cite[Lemma 3.5]{Beck13b}, the above inequality implies that
\begin{align}
\Delta^{(r+1)}&\le 3\max\left\{\Delta^0  , {2C^2}\right\}\frac{1}{r+1}\nonumber\\
&\le 3\max\left\{\Delta^0  , 2\log^2(2NK)\left(P_{\max}+\frac{L^2}{P_{\min}}\right) R^2_0\right\}\frac{1}{r+1}
\end{align}
When $P_k=L$ for all $k$, the bound reduces to
\begin{align}
\Delta^{(r+1)}&\le 3\max\left\{\Delta^0  , 4\log^2(2NK)R^2_0 L\right\}\frac{1}{r+1}
\end{align}
When the problem is smooth and unconstrained, we have
$$\Delta^{(0)}\le \langle \nabla f(x^{(0)}), x^{(0)}-x^* \rangle = \langle \nabla f(x^{(0)})-f(x^*), x^{(0)}-x^* \rangle\le L\|x^{(0)}-x^*\|^2.$$
This completes the proof. \QED
\end{proof}

\vspace{-0.2cm}
We comment on the bounds derived in the above theorem.
The bound for BCPG with uniform ``conservative" stepsize $1/L$ has the same order as the GD method, except for the $\log^2(2NK)$ factor (cf. \eqref{eq:complexity:gd}). In \cite[Corollary 3.2]{Beck13}, it is shown that the BCGD with the same ``conservative" stepsize achieves a sublinear rate with a constant of $4L(1+K) R^2_0$, which is about $K/(3\log^2(2NK))$ times worse than our bound. Further, our bound has the same dependency on $L$ (i.e., $12L$ v.s. $L/2$) as the one derived in \cite{Saha10} for BCPG with a ``conservative" stepsize to solve an $\ell_1$ penalized quadratic problem with special data matrix, but our bound holds true for a much larger class of problems (i.e., all quadratic nonsmooth problem in the form of \eqref{eq:qp:ns}). However, in practice such conservative stepsize is slow (compared with BCPG with $P_k=L_k$, for all $k$) hence is rarely used.


The bounds derived in Theorem \ref{thm:bound:bcpg} is again at least $K/\log^2(2NK)$ times better than existing bounds of cyclic BCPG. For example, when the problem is smooth and unconstrained, the ratio between our bound \eqref{eq:bcpg:bound:uncons} and the bound \eqref{eq:complexity:bcd} is given by{
\begin{align}
\frac{ 6 R^2_0\log^2(2NK)(L^2/L_{\min}+L_{\max})}{ C L_{\max}(1+K L^2/L^2_{\min}) R^2_0} \le \frac{6\log^2(2NK)(1+L^2/(L_{\min}L_{\max}))}{C(1+KL^2/L^2_{\min})}= \cO(\log^2(2NK)/K)
\end{align}}
\!\!where in the last inequality we have used the fact that $L_{\max}/L_{\min}\ge 1$. 

{For unconstrained smooth problems, let us compare the bound derived in the second part of Theorem \ref{thm:bound:bcpg} (stepsize $P_k = L_k, \forall k$) with that of the GD \eqref{eq:complexity:gd}. If $L = KL_k$ for all $k$ (problem badly conditioned), our bound is about $K \log^2(2NK)$ times worse than that of the GD.
 This indicates a counter-intuitive phenomenon: by choosing conservative stepsize $P_k = L, \forall k$ the iteration complexity of BCGD
 is $K$ times better compared with choosing a more aggressive stepzise $P_k  = L_k, \forall k$.
 It also indicates that the factor $L/L_{\min}$ may hide an additional factor of $K$.}

\section{Improved Bounds of Cyclic BCD for Quadratic Problems}\label{sec:BCD}
In the previous section, we analyzed an inexact cyclic BCD algorithm, the BCPG algorithm. In this section we analyze the performance of the cyclic BCD algorithm (with exact minimization, cf. \eqref{eq:BCM}), when applied to the quadratic problem \eqref{eq:qp:ns}.

We will divide our analysis into three cases ($\lambda_{\min}$ denotes the minimum eigenvalue):
\begin{enumerate}
  \item Each $A_k$ has full column rank with $\lambda_{\min}(A^T_k A_k)\ge \sigma^2_k$, where $\sigma_k>0$ is some known constant;
  \item Each $A_k$ has full row rank with $\lambda_{\min}(A_k A^T_k)\ge \gamma_k^2$, where $\gamma_k>0$ is some known constant;
  \item Each $A_k$ has neither full column rank nor full row rank.
\end{enumerate}
Note that in the last two cases the subproblems are not strongly convex hence may not have unique solutions. Without uniqueness, the only known iteration complexity bound for the cyclic BCD is developed in \cite{hong13complexity}, but such bound is at least proportional to $K^2$.

Our analysis again follows the three-step approach used in the previous section.
To estimate the descent, we have the following lemma.
\begin{lemma}\label{lm:descent:bcd}
 We have the following estimate of the descent when using the BCD
\begin{align}
f(x^{(r)})-f(x^{(r+1)})\ge \frac{1}{2} (x^{(r+1)}-x^{(r)})^T\tA^T \tA (x^{(r+1)}-x^{(r)})\label{eq:descent}.
\end{align}
\end{lemma}
\begin{proof} 
We have the following series of inequalities
\begin{align*}
&f(x^{(r)})-f(x^{(r+1)})\\
& = \sum_{k=1}^{K} f(w_k^{(r+1)})-f(w^{(r+1)}_{k+1})\\
& \stackrel{\rm (i)}= \sum_{k=1}^{K}\langle \nabla_k g(w^{(r+1)}_{k+1}), x^{(r)}_k-x^{(r+1)}_k \rangle+\frac{1}{2}\sum_{k=1}^{K}\|A_k (x^{(r+1)}_k-x^{(r)}_k)\|^2 + \sum_{k=1}^{K}h_k(x_k^{(r)})-h_k(x_k^{(r+1)})\\
&\stackrel{\rm(ii)}\ge \frac{1}{2}\sum_{k=1}^{K}\|A_k (x^{(r+1)}_k-x^{(r)}_k)\|^2\\
&=\frac{1}{2} (x^{(r+1)}-x^{(r)})^T\tA^T \tA (x^{(r+1)}-x^{(r)})\nonumber
\end{align*}
where in $\rm (i)$ we have used the fact that the second order Taylor expansion of a quadratic problem is exact; in $\rm(ii)$ we have used the optimality of $w_{k+1}^{(r+1)}$; cf. \eqref{eq:opt:bcpg}. \QED
\end{proof}

The cost-to-go estimate is given by the following lemma.
\begin{lemma}\label{lemma:cost:bcd}
 Let $\sigma_{\min}:=\min_i\{\sigma_i\}$ and $\gamma_{\min}:=\min_i\{\gamma_i\}$ . We have the following estimate of the optimality gap when using the cyclic BCD to solve \eqref{eq:qp:ns}.
 \begin{enumerate}
   \item When each $A_k$ has full column rank, we have
   \begin{align}
  \|f(x^{(r+1)})-f(x^*)\|&\le R_0\frac{1}{\sigma_{\min}}\log(2NK)\left( L   +L_{\max}\right)\left\|(x^{(r+1)}-x^{(r)})(\Sigma\otimes I_N)\right\|\label{eq:cost:case1}
\end{align}
   \item When each $A_k$ has full row rank, we have
      \begin{align}
  \|f(x^{(r+1)})-f(x^*)\|&\le R_0 \left(\frac{1}{\gamma_{\min}}\log(2NK)(L + {L_{\max}})\right)\left\|(x^{(r+1)}-x^{(r)})\tA^T\right\|\label{eq:cost:case2}
\end{align}
   \item When each $A_k$ has neither full column rank nor full row rank, we have
      \begin{align}
  \|f(x^{(r+1)})-f(x^*)\|&\le R_0\sqrt{L_{\max}}\left(K+2\right)\left\|(x^{(r+1)}-x^{(r)})\tA^T\right\|\label{eq:cost:case3}
\end{align}
 \end{enumerate}

\end{lemma}
\begin{proof}
We again bound $\Delta^{(r+1)}=f(x^{(r+1)})-f(x^*)$. First we have{
\begin{align}
&f(x^{(r+1)})-f(x^*)\nonumber\\
& \le \langle \nabla g(x^{(r+1)}), x^{(r+1)}-x^* \rangle +\sum_{k=1}^{K}h_k(x^{(r+1)}_k)-h_k(x^*_k)\nonumber\\
 &\stackrel{\rm (i)}\le\sum_{k=1}^{K}\langle \nabla_k g(x^{(r+1)}) -\nabla_k g(w^{(r+1)}_{k+1}), x^{(r+1)}_k-x^*_k  \rangle\nonumber\\
 &=\sum_{k=1}^{K}\left\langle \left(\sum_{j>k}A_j(x^{(r+1)}_j-x^{(r)}_j)\right), A_k(x^{(r+1)}_k-x^*_k)  \right\rangle\nonumber\\
    &\stackrel{\rm (ii)}= (x^{(r+1)}-x^{(r)})^T\tA^T (D_3\otimes I_M) \tA (x^{(r+1)}-x^*)\nonumber
    \label{eq:cost-to-go1}
\end{align}}
\!\!where in $\rm (i)$ we have used the optimality condition of $w_{k+1}^{(r+1)}$ (cf. \eqref{eq:opt:bcpg}); in $\rm (ii)$ $D_3$ is the following lower triangular matrix
\begin{align}
D_3:= \left[\begin{array}{llllll}
0 & 0 &  0& \cdots& 0 & 0\\
1 & 0& 0& \cdots& 0 &0\\
1& 1& 0& \cdots& 0&0\\
\vdots&\vdots&\vdots&\cdots&\vdots&\vdots \\
1& 1& 1& \cdots& 1& 0
\end{array}\right]\in\mathbb{R}^{K\times K}.
\end{align}
Below we bound $\|f(x^{(r+1)})-f(x^*)\|$ for three different cases.

First, suppose that $\lambda_{\min}(A^T_k A_k)\ge \sigma^2_k$ for all $k$. Define $\Sigma: = \mbox{diag}[\sigma_1,\cdots, \sigma_k]\in\mathbb{R}^{K\times K}$. We have
\begin{align}
  \|f(x^{(r+1)})-f(x^*)\|&\le R_0\left\|(\Sigma^{-1}\otimes I_N)\left(\tA^T(D_3\otimes I_M)\tA\right)\right\|\left\|(x^{(r+1)}-x^{(r)})(\Sigma\otimes I_N)\right\|\nonumber\\
  &= R_0\left\|(\Sigma^{-1}\otimes I_N)\left((A^T A-\tA^T\tA) \odot D_2\right)\right\|\left\|(x^{(r+1)}-x^{(r)})(\Sigma\otimes I_N)\right\|\nonumber\\
  &\le R_0\left\|(x^{(r+1)}-x^{(r)})(\Sigma\otimes I_N)\right\|\frac{1}{\sigma_{\min}}\log(2NK)\left( L   +L_{\max}\right)
\end{align}
where $D_2$ is the $KN\times KN$ lower triangular matrix; in the last inequality we have again used the property of lower triangular truncation operator; see the proof of Lemma \ref{lemma:cost:bcpg}.

Second, suppose that $\lambda_{\min}(A_k A^T_k)\ge \gamma^2_k$ for all $k$. Define $\Gamma: = \mbox{diag}[\gamma_1,\cdots, \gamma_k]\in\mathbb{R}^{K\times K}$. We have
\begin{align}
  &\|f(x^{(r+1)})-f(x^*)\|\nonumber\\
  &\le R_0\left(\left\|(D_3\otimes I_M)\tA\right\|\right)\left\|(x^{(r+1)}-x^{(r)})\tA^T\right\|\nonumber\\
  &\stackrel{\rm(i)}\le R_0\left\|(x^{(r+1)}-x^{(r)})\tA^T\right\| \left\|\tA^{\dagger}\tA^T(D_3\otimes I_M)\tA\right\|\nonumber\\
  &\stackrel{\rm(ii)}\le R_0\left\|(x^{(r+1)}-x^{(r)})\tA^T\right\| \left(\|\tA^{\dagger}\|\log(2NK)(L+{L_{\max}}) \right)\nonumber\\
  &\le R_0\left\|(x^{(r+1)}-x^{(r)})\tA^T\right\| \left(\frac{1}{\gamma_{\min}}\log(2NK)(L + {L_{\max}})\right)\nonumber
\end{align}
where in ${\rm (i)}$ $\tA^{\dagger}$ is the Moore-Penrose pseudoinverse of $\tA$, and we have utilized the fact that when $A_kA^T_k$ is invertible, its  Moore-Penrose pseudoinverse  is given by $(A_kA^T_k)^{-1}A_k$; in ${\rm (ii)}$ we again have used the property of lower triangular truncation operator; see the proof of Lemma \ref{lemma:cost:bcpg}.

Third, for general $A_k$'s, we have
\begin{align}
  &\|f(x^{(r+1)})-f(x^*)\|\nonumber\\
  &\le R_0\left\|(D_3\otimes I_M)\tA\right\|\left\|(x^{(r+1)}-x^{(r)})\tA^T\right\|\nonumber\\
  &\le R_0\sqrt{L_{\max}}\left(K+1\right)\left\|(x^{(r+1)}-x^{(r)})\tA^T\right\|\nonumber.
\end{align}
This completes the proof.
\end{proof}


The third step again combines the previous two steps to derive the final bounds.
\begin{theorem}\label{thm:bound:bcd}
The iteration complexity of using cyclic BCD to solve \eqref{eq:qp:ns} is given below.

 \begin{enumerate}
   \item When each $A_k$ has full column rank, we have
   \begin{align}
\Delta^{(r+1)}\le3\max\left\{\Delta^0, 2R^2_0\frac{1}{\sigma^2_{\min}}\log^2(2NK)\left( L^2   +L^2_{\max}\right)\right\}\frac{1}{r+1}
\end{align}
   \item When each $A_k$ has full row rank, we have
      \begin{align}
\Delta^{(r+1)}\le3\max\left\{\Delta^0, 2R^2_0\left( \frac{1}{\gamma^2_{\min}}\log^2(2NK)(L^2   +L^2_{\max})\right)\right\}\frac{1}{r+1}
\end{align}
   \item When each $A_k$ has neither full column rank nor full row rank, we have
      \begin{align}
\Delta^{(r+1)}\le 3\max\left\{\Delta^0, 2R^2_0L_{\max}\left(1+K^2\right)\right\}\frac{1}{r+1}
\end{align}
 \end{enumerate}

\end{theorem}
\begin{proof}
First, consider the case where $A_k$'s all have full column rank. In this case the descent estimate \eqref{eq:descent} can be further expressed as
\begin{align}
\Delta^{(r)}-\Delta^{(r+1)}\ge \sum_{k=1}^{K}\frac{1}{2\sigma^2_{k}} \|x_k^{(r+1)}-x_k^{(r)}\|^2=\left\|(x_k^{(r+1)}-x_k^{(r)})^T (\Sigma^2\otimes I_N) (x_k^{(r+1)}-x_k^{(r)})\right\|^2.
\end{align}
Using this relation and squaring both sides of \eqref{eq:cost:case1}, we obtain
\begin{align}
  \|\Delta^{(r+1)}\|^2&\le 4R^2_0\frac{1}{\sigma^2_{\min}}\log^2(2NK)\left( L^2   +L^2_{\max}\right)(\Delta^{(r)}-\Delta^{(r+1)})\label{eq:recursion:1}
\end{align}

Second, when $A_k$'s all have full row rank, we can simply combine \eqref{eq:descent}  and \eqref{eq:cost:case2} and obtain
\begin{align}
  \|\Delta^{(r+1)}\|^2&\le 4R^2_0\frac{1}{\gamma^2_{\min}}\log^2(2NK)\left( L^2   +L^2_{\max}\right)(\Delta^{(r)}-\Delta^{(r+1)})\label{eq:recursion:2}
\end{align}

Third, when $A_k$'s all neither full row rank nor full column rank, we can combine \eqref{eq:descent}  and \eqref{eq:cost:case3} and obtain
\begin{align}
  \|\Delta^{(r+1)}\|^2&\le 4R_0^2L_{\max}\left(1+K^2\right)(\Delta^{(r)}-\Delta^{(r+1)})\label{eq:recursion:3}
\end{align}

Again by utilizing  result from \cite[Lemma 3.5]{Beck13b} to the recursions \eqref{eq:recursion:1}--\eqref{eq:recursion:3} we obtained the desired results.  \QED

\end{proof}

We note that for the special case where $X_k\in\mathbb{R}$ for all $k$ (scalar variable block), and $A_k\ne 0$ for all $k$, then each $A_k$ is a scalar and we have $L_k=\|A^T_k A_k\|$ for all $k$. In this case, $\sigma^2_{\min}=\delta^2_{\min}=L_{\min}$, and we can take the minimum of the three bounds presented in Theorem \ref{thm:bound:bcd}, which has the form
      \begin{align}
\Delta^{(r+1)}\le3\max\left\{\Delta^0, 2R^2_0 \log^2(2NK)\left(\frac{L^2}{L_{\min}}   +\frac{L^2_{\max}}{L_{\min}}\right)\right\}\frac{1}{r+1}
\end{align}
Except for the log factor, such bound is again no longer explicitly  dependent on $K$, and is about $K/\log^2(2NK)$ times better than the existing bound on the cyclic BCD \cite{hong13complexity}.


\subsection{Discussion on the Tightness in terms of $K$}   
Viewing $L/L_{\min}$ and $R_0^2$ as constants,
 the bound derived in \eqref{eq:bcpg:bound:uncons} for BCPG is $\mathcal{O}(\log^2(2K){L}/{r})$.
An immediate question is: is it possible to further reduce these bounds by an additional order of $K$, i.e. $\mathcal{O}(\log^2(2K){L}/{(K r) })$?
The answer is negative at least for $r = 1$, as shown below.
Note that the analysis below does not exclude the possibility that BCD can achieve an iteration complexity $ \mathcal{O}({L}/{ r^2 })  $
or $\mathcal{O}(c^r )  $, where $c<1$ is a constant depending on $L$.
Finding the iteration complexity lower bound of exact BCD for smooth quadratic minimization remains an interesting open question.

Consider the following problem
\begin{align}
\min \quad g(x):=\|A x\|^2\label{eq:example}
\end{align}
where ${A}\in\mathbb{R}^{K\times K}$ is a square tridiagonal Toeplitz matrix of the following form
\begin{align}\label{eq:A}
A=\left [ \begin{array}{llllllll} 1& 1 & 0& 0& \cdots & 0& 0 & 0\\
1& 1 & 1& 0& \cdots & 0& 0 & 0\\
0& 1 & 1& 1& \cdots &0 & 0 & 0\\
\vdots& \vdots & \vdots& \vdots& \cdots & \vdots & \vdots & \vdots\\
0& 0 & 0& 0& \cdots &1 & 1 & 1\\
0& 0 & 0& 0& \cdots &0 & 1 & 1
\end{array}\right].
\end{align}
Note that $A^T A\succeq 0$, but it is not necessarily full rank. Clearly $x^*=[0,\cdots, 0]^T$ is one of the optimal solutions for problem \eqref{eq:example}, while $0$ is also the the optimal objective value. Further, we know that the maximum eigenvalue for the matrix $A$ is:
$$\lambda_{\max}(A)=1+2 \cos(i \pi/(K+1)), \quad k=1,\cdots, K,$$
which means that $\|A\|_2\le 3$. Therefore the Lipschitz constant of $\nabla g(x)$ is bounded by  $L\le 18$, {\it regardless of the dimension $K$}.
For the classical gradient decent method, we must have (cf. \eqref{eq:complexity:gd}){
\begin{align}
\Delta^{(r)}_{\rm GD}\le \frac{2\times 18 R^2_0}{r+4}.
\end{align}}
Let each $x_k\in\mathbb{R}$ be a {\it scalar}. We can equivalently write problem \eqref{eq:example} as{
\begin{align}
\min \quad g(x):=\big\|\sum_{k=1}^{K}A_k x_k\big\|^2\label{eq:example:bcd}
\end{align}}
\!where $A_k$ is the $k$th column of $A$. For this block structure, we have $L_{\min}=4$, $L_{\max}=9$, and $L/L_{\min}\le 5$.

Based upon the above block structure, it is easy to verify that the cyclic BCD with exact minimization generates the same sequence as the cyclic BCPG with stepsizes $\{1/L_k\}_{k}$. Therefore our result below holds for {\it both} algorithms.
 Consider problem \eqref{eq:example:bcd} with $A=[A_1,\cdots, A_K]$ given by \eqref{eq:A}. Let
\begin{align}
x^{(0)} = \big[ 1, \frac{1}{8}, \frac{3}{4}, 1, 1, \cdots, 1, 1 \big]^T\in\mathbb{R}^K. \label{eq:x0}
\end{align}
We can show through explicit computation that after running cyclic BCD/BCPG for one pass of all blocks, the optimality gap $\Delta^{(1)}$ is bounded below by
\begin{align}\label{eq:lower:bound}
{\Delta^{(1)}} \ge \frac{9(K-3)}{4(K-1)}{\|x^{(0)}-x^*\|^2}, \; \forall~K\ge 3.
\end{align}
The derivation is relegated to the Appendix.
This result implies that it is not possible to derive a global complexity bound $\mathcal{O}(\frac{L}{Kr })$ for BCD and BCPG with stepsizes $\{1/L_k\}_{k}$

\section{Iteration Complexity for General Convex Problems}\label{sec: general convex}
In this section, we consider improved iteration complexity bounds of BCD for general unconstrained smooth convex problems. We prove a general iteration complexity result, which includes a result of Beck et al. \cite{Beck13} as a special case. Our analysis for the general case also applies to smooth quadratic problems, but is very different from the analysis in previous sections
 for quadratic problems. For simplicity, we only consider the case $N=1$ (scalar blocks); the generalization to the case $N>1$ is left as future work.

Let us assume that the smooth objective $g$ has second order derivatives $H_{ij}(x): = \frac{ \partial^2 g }{\partial x_i \partial x_j }( x ) $.
When each block is just a coordinate, we assume
$ |H_{ij}(x)| \leq L_{ij}, \forall i,j .$
Then $L_i = L_{ii}$ and $L_{ij}\leq \sqrt{L_i} \sqrt{L_j} $.
For unconstrained smooth convex problems with scalar block variables, the BCPG iteration reduces to the following coordinate gradient descent (CGD) iteration:{
\begin{equation}\label{chain update}
  x^{(r)} = w_1^{(r)} \overset{d_1 }{\longrightarrow} w^{(r)}_2 \overset{d_2}{\longrightarrow }w^{(r)}_3 \longrightarrow \dots \overset{d_K}{\longrightarrow}
  w^{(r)}_{K+1} = x^{(r+1)},
\end{equation}}
where $ d_k = \nabla_k g(w^{(r)}_k) $ and $  w^{(r)}_k \overset{d_k }{\longrightarrow} w^{(r)}_{k+1} $ means that
$w^{(r)}_{k+1}$ is a linear combination of $w^{(r)}_k $ and $d_k e_k$ ($e_k$ is the $k$-th block unit vector).

The general framework follows the standard three-step approach that combines sufficient descent and cost-to-go estimate; nevertheless, the analysis of the sufficient descent is very different from the methods used in the previous sections.
\begin{lemma}\label{lemma:descent:general}(sufficient descent)
There exist $\xi_k\in\mathbb{R}^{N}$ lying in the line segment between $x^{(r)}$ and $w^{(r)}_k, k=2,\dots, K$, such that
{
    \begin{equation}\label{eq: sufficient decrease}
      g(x^{(r)}) - g(x^{(r+1)})  \geq \frac{1}{2} \frac{1}{ P_{\max} +  \frac{ \| H(\xi)\|^2 }{ P_{\min} }  }   \| \nabla g(x^{(r)})\|^2,
\end{equation} }
where {
  \begin{equation}\label{H def, first}
  H(\xi) \triangleq
     \begin{bmatrix}
  0                           &  0   & 0  &  \dots  &  0 &  0   \\
  H_{21}(\xi_2) &  0   & 0  &  \dots  &  0 &  0   \\
  H_{31}(\xi_3) &  H_{32}(\xi_3)   & 0  &  \dots  &  0 &  0   \\
    \vdots                    &  \vdots   &  \vdots  &  \vdots   &  \ddots  &  \vdots  \\
       H_{K1}(\xi_K) &   H_{K2}(\xi_K)   &    H_{K3}(\xi_K)   &  \dots  &   H_{K,K-1}(\xi_K) &  0   \\
\end{bmatrix}.
\end{equation}}
\end{lemma}
\begin{proof}
Since $ w^r_{k+1}$ and $w^r_{k}$ only differ by the $k$-th block, and $\nabla_k g$ is Lipschitz continuous with Lipschitz constant $L_k$, we have
\footnote{ A stronger bound is $g(w^r_{k+1}) \leq g(w^r_k ) -   \frac{1}{ 2 P_k } \| \nabla_k g(w^r_{k}) \|^2 $,
where
$
 \hat{P}_k =   \frac{P_k^2  }{ 2P_k - L_k}  \leq  P_k ,
$ but since $ P_k \leq 2 P_k - L_k  \leq 2P_k$, the improvement ratio of using this stronger bound is no more than a factor of $2$.
}
\begin{equation}
\begin{split}
  g(w^r_{k+1}) \leq  & g(w^r_k ) + \langle \nabla_k g(w^r_{k}), w^r_{k+1} - w^r_k  \rangle + \frac{L_k}{2} \| w^r_{k+1} - w^r_k \|^2  \\
               =     & g(w^r_k ) -   \frac{ 2P_k - L_k }{ 2 P_k^2 } \| \nabla_k g(w^r_{k}) \|^2 \\
               \leq     & g(w^r_k ) -   \frac{1}{ 2 P_k } \| \nabla_k g(w^r_{k}) \|^2 ,
\end{split}
\end{equation}
where the last inequality is due to $P_k \geq L_k $.

The amount of decrease of BCGD is
\begin{equation}\label{naive decrease of BCGD}
  g(x^r) - g(x^{r+1})
   = \sum_{k=1}^r [ g(w^r_k )  - g( w^{r}_{k+1} )  ]
   \geq \sum_{k=1}^r \frac{1}{2 P_k } \| \nabla_k g(w^r_{k}) \|^2.
\end{equation}


Since $$ w^{r}_k = x^r - \begin{bmatrix}
\frac{1}{ P_1 } d_1 , \dots ,  \frac{1}{ P_{k-1} } d_{k-1} ,  0 ,  \dots ,  0
\end{bmatrix}^T,
 $$
by the mean-value theorem, there must exist $\xi_k $ such that
\begin{equation}\label{nabla k diff, first bound}
\begin{split}
     &  \nabla_k g( x^r ) - \nabla_k g( w^{r}_k )
       =     \nabla( \nabla_k g )  (\xi_k )  \cdot ( x^r - w^r_k  )  \\
      & =   \left[ \frac{\partial^2 g }{\partial x_k \partial x_1 } (\xi_k ), \dots, \frac{\partial^2 g }{\partial x_k \partial x_{k-1} }(\xi_k )
     , 0, \dots, 0 \right]
      \begin{bmatrix}
\frac{1}{ P_1 } d_1 , \dots ,  \frac{1}{ P_{k-1} } d_{k-1} ,  0 ,  \dots ,  0
\end{bmatrix}^T   \\
& = \left[ \frac{ 1 }{ \sqrt{P_1} } H_{k1} (\xi_k ), \dots, \frac{ 1 }{ \sqrt{P_{k-1}} } H_{k,k-1}(\xi_k )
     , 0, \dots, 0 \right]
      \begin{bmatrix}
 \frac{1}{\sqrt{P_1} } d_1 ,  \dots ,  \frac{1}{ \sqrt{P_K } } d_{K}
\end{bmatrix}^T ,
\end{split}
\end{equation}
where $H_{ij}(x) = \frac{ \partial^2 g }{\partial x_i \partial x_j }( x ) $ is the second order derivative of $g$.
Then
\begin{equation}\label{nabla k first bound}
\begin{split}
  \nabla_k g( x^r ) & = \nabla_k g( x^r ) - \nabla_k g( w^{r}_k ) + \nabla_k g( w^{r}_k ) \\
   & = \left[ \frac{ 1 }{ \sqrt{P_1} } H_{k1} (\xi_k ), \dots, \frac{ 1 }{ \sqrt{P_{k-1}} } H_{k,k-1}(\xi_k )
     , 0, \dots, 0 \right]
      \begin{bmatrix}
 \frac{1}{\sqrt{P_1} } d_1 ,  \dots ,  \frac{1}{ \sqrt{P_K } } d_{K}
\end{bmatrix}^T  +  d_k  \\
  & =  \left[ \frac{ 1 }{ \sqrt{P_1} } H_{k1} (\xi_k ), \dots, \frac{ 1 }{ \sqrt{P_{k-1}} } H_{k,k-1}(\xi_k )
     , \sqrt{P_k}, 0, \dots, 0 \right]
      \begin{bmatrix}
 \frac{1}{\sqrt{P_1} } d_1 ,  \dots ,  \frac{1}{ \sqrt{P_K } } d_{K} \end{bmatrix}^T  \\
 &  = v_k^T d,
\end{split}
\end{equation}
where we have defined
\begin{equation}
\begin{split}
d \triangleq  &  \begin{bmatrix}
 \frac{1}{\sqrt{P_1} } d_1 ,  \dots ,  \frac{1}{ \sqrt{P_K } } d_{K}
\end{bmatrix}^T,  \\
v_k    \triangleq  & \left[ \frac{ 1 }{ \sqrt{ P_{1}} } H_{k1} (\xi_k ),   \dots, \frac{ 1 }{ \sqrt{P_{k-1}} } H_{k,k-1}(\xi_k )
     , \sqrt{P_k} , \dots, 0 \right]  .
\end{split}
\end{equation}
Let{\small
\begin{equation}\label{V def}
  V \triangleq  \begin{bmatrix} v_1^T \\ \dots \\ v_K^T \end{bmatrix}
     =  \begin{bmatrix}
   \sqrt{P_1}                        &  0   & 0  &  \dots  &  0 &  0   \\
  \frac{ 1}{\sqrt{ P_1 } } H_{21}(\xi_2) & \sqrt{P_2}   & 0  &  \dots  &  0 &  0   \\
  \frac{ 1}{\sqrt{  P_1 } } H_{31}(\xi_3) & \frac{ 1}{\sqrt{ P_2 } } H_{32}(\xi_3)   &  \sqrt{P_3}   &  \dots  &  0 &  0   \\
  \frac{ 1}{\sqrt{  P_1 } } H_{41}(\xi_4) &  \frac{ 1}{\sqrt{  P_2 } } H_{42}(\xi_4)   &
 \frac{ 1}{\sqrt{  P_3 } } H_{43}(\xi_4)   &  \ddots  &  0 &  0   \\
    \vdots                    &  \vdots   &  \vdots  &  \vdots   &  \ddots  &  \vdots  \\
    \frac{ 1}{\sqrt{  P_1 } }  H_{K1}(\xi_K) &   \frac{ 1}{\sqrt{  P_2 } } H_{K2}(\xi_K)   &    \frac{ 1}{\sqrt{  P_3 } } H_{K3}(\xi_K)   &  \dots  &  \frac{ 1}{\sqrt{  P_{K-1} } } H_{K,K-1}(\xi_K) &  \sqrt{P_K}    \\
\end{bmatrix}
\end{equation}}
Therefore, we have
$$  \| \nabla g(x^r) \|^2 =  \sum_k  \| \nabla_k g( x^r ) \|^2
  \overset{\eqref{nabla k first bound}}{=}   \sum_k  \| v_k^T d \|^2 = \| V d \|^2 \leq \| V\|^2 \| d \|^2
   = \|V \|^2 \sum_k \frac{1}{P_k} \| \nabla_k g(w^{r}_k )\|^2.
  $$
Combining with \eqref{naive decrease of BCGD}, we get
\begin{equation}\label{use V to bound decrease}
g(x^r) - g(x^{r+1})
   \geq \sum_k \frac{1}{2 P_k } \| \nabla_k g(w^r_{k}) \|^2
   \geq  \frac{1}{ 2 \| V\|^2 }   \| \nabla g(x^r)\|^2.
\end{equation}

Let $ D \triangleq \text{Diag}( P_1, \dots, P_K ) $ and let $H(\xi)$ be defined as in \eqref{H def, first},
then $V = D^{1/2} + H(\xi) D^{-1/2} $, which implies
$$
  \| V \|^2 = \| D^{1/2} + H(\xi) D^{-1/2}  \|^2  \leq 2( \| D^{1/2}\|^2 + \| H(\xi) D^{-1/2}  \|^2 ) \leq 2\left( P_{\max} +  \frac{\| H(\xi) \|^2 }{ P_{\min} } \right) .
$$
Plugging into \eqref{use V to bound decrease}, we obtain the desired result. \QED
\end{proof}

The above lemma provides an estimate of the sufficient descent. The intuition is that CGD can be viewed as an inexact gradient descent method, thus the amount of descent can be bounded in terms of the norm of the full gradient.
It would be difficult to further tighten this bound if the goal is to obtain a sufficient descent based on the norm of the full gradient. To further improve this bound, we suspect that either the third order derivatives of $g$ or the relation between $ H(\xi) $ and $d_k = \nabla_k g(w^r_k) $ should be considered 
 (in the derivation of Lemma \ref{lemma:descent:general} we treat $H(\xi)$ and $d_k$ independently).

Having established the sufficient descent in terms of the full gradient $  \nabla g(x^{(r)})$,
we can easily prove the iteration complexity result, following the standard analysis of GD (see, e.g. \cite[Theorem 2.1.13]{Nesterov04}).
\begin{theorem}\label{thm general cvx} Suppose $\beta \in \mathbb{R}_+$ satisfies {
\begin{equation}\label{def of beta}
 \| H (\xi ) \| \leq \beta, \ \forall \ \xi_2 , \dots, \xi_K \in \mathbb{R}^{N},
\end{equation} }
then we have
 {
    \begin{equation}\label{eq in Thm 1, general ite complexity of BCGD}
     g(x^{(r)}) - g(x^*)   \leq 2\left( P_{\max} +  \frac{ \beta^2 }{ P_{\min} } \right)   \frac{ R_0^2 }{r}, \ \forall \ r \geq 1.
\end{equation} }
\end{theorem}
\begin{proof}
Denote $ \omega = \frac{1}{2} \frac{1}{ P_{\max} +  \frac{ \beta^2 }{ P_{\min} }  }   $,
then \eqref{eq: sufficient decrease} becomes
\begin{equation}\label{successive diff, second time}
      g(x^{(r)}) - g(x^{(r+1)})  \geq  \omega \| \nabla g(x^{(r)})\|^2 , \ \forall r.
\end{equation}
This relation also implies $g(x^{(r)}) \leq g(x^{(0)})$, thus by the definition of $R_0 $ in \eqref{eq:r0} we have
$\| x^{(r)} - x^* \| \leq R_0 $.
By the convexity of $g$ and the Cauchy-Schwartz inequality, we have
$$
 g(x^{(r)}) - g(x^*) \leq \langle \nabla g(x^{(r)}), x^{(r)} - x^*  \rangle \leq \|\nabla g(x^{(r)})\| R_0.
$$
Combining with \eqref{successive diff, second time}, we obtain
$$
  g(x^{(r)}) - g(x^{(r+1)}) \geq \frac{ \omega }{ R_0^2 } ( g(x^{(r)}) - g(x^*) )^2.
$$
Let $ \Delta^{(r)} = g(x^{(r)}) - g(x^*)$, we obtain
$$
  \Delta^{(r)} - \Delta^{(r+1)} \geq \frac{ \omega }{ R_0^2 } \Delta^{(r)}.
$$
Following the proof of , we have
$$
  \frac{1}{\Delta^{(r+1)}} \geq \frac{1}{\Delta^{(r)}}  + \frac{\omega }{ R_0^2 } \frac{\Delta^{(r)}}{\Delta^{(r+1)}}
  \geq \frac{1}{\Delta^{(r)}}  + \frac{\omega }{ R_0^2 }.
$$
Summarizing the inequalities, we get
$$
  \frac{1}{\Delta^{(r+1)}}  \geq \frac{1}{\Delta^{(0)}} + \frac{\omega }{ R_0^2 } (r+1) \geq \frac{\omega }{ R_0^2 } (r+1),
$$
which leads to
$$
 \Delta^{(r+1)} = g(x^{(r+1)}) - g(x^*) \leq \frac{1}{\omega} \frac{R_0^2}{ r+1 } = 2( P_{\max} +  \frac{ \beta^2 }{ P_{\min} } ) \frac{R_0^2}{ r+1 }.
$$
This completes the proof of the claim. \QED
\end{proof}
In the following corollary, we provide an initial estimate of  $\beta$ and obtain a more concrete iteration complexity bound.
\begin{corollary}\label{coro: CGD ite complexity} For CGD with  $P_k  \geq L_{\max}, \forall k$, we have $\beta^2 \leq \min \{ KL^2, (\sum_k L_k)^2 \}$, thus
{
    \begin{equation}\label{relate BCGD to GD}
     g(x^{(r)}) - g(x^*)   \leq  2 \left( P_{\max} +  \frac{ \min\{ K L^2, (\sum_k L_k )^2  \}  }{ P_{\min} } \right)   \frac{ R_0^2 }{r} , \ \forall \ r \geq 1.
\end{equation}}
\end{corollary}
\begin{proof}
First, from the fact that $H_{kj}(\xi_k)$ is a scalar bounded above by $ | H_{kj}(\xi_k ) | \leq L_{kj} \leq \sqrt{L_k L_j }$, thus{
\begin{equation}\label{H first bound}
\begin{split}
 \| H \|^2 \leq \| H\|_F^2 = \sum_{k < j} | H_{kj}(\xi_k ) |^2 \leq \sum_{k < j}L_k L_j
 \leq (\sum_k  L_k)^2.
\end{split}
\end{equation}}
We provide the second bound of $\| H \|$ below. Let $H_k $ denote the $k$-th row of $H$, then
$\| H_k \| \leq L$. Therefore, we have
$$
  \| H\|^2 \leq \| H\|_F^2 = \sum_k \| H_k\|^2 \leq \sum_k L^2 =  K L^2.
$$
Plugging this bound and \eqref{H first bound} into \eqref{eq in Thm 1, general ite complexity of BCGD}, we obtain the desired result. \QED
\end{proof}

Let us compare this bound with the bound derived in \cite[Theorem 3.1]{Beck13} (replacing the denominator $r + 8/K$ by $r$), which is
\begin{equation}\label{Beck}
     g(x^r) - g(x^*)   \leq  4 \left( P_{\max} +  \frac{P_{\max}}{P_{\min}} \frac{  K L^2  }{ P_{\min} } \right)   \frac{ R^2 }{r} .
\end{equation} 
The key step in establishing \eqref{Beck} is essentially proving $ \beta^2 \leq KL^2$ (though not explicitly stated in their paper).
In our new bound, besides reducing the coefficient from $4$ to $2$ and removing the factor $  \frac{P_{\max}}{P_{\min}} $,
 we provide a new bound for $\beta^2 $ which is $(\sum_k L_k )^2 $, thus replacing
$K L^2 $ by  $ \min\{ K L^2, (\sum_k L_k )^2  \}  $.
Neither of the two bounds $K L^2$ and $(\sum_k L_k)^2$ implies the other: when $L = L_k, \forall k$ the new bound $(\sum_k L_k)^2$ is $K$ times larger;
when $L = K L_k, \forall k$ or $L = L_1 > L_2 = \dots = L_K = 0$ the new bound is $K$ times smaller.
In fact, when $L = K L_k, \forall k$, our new bound is $K$ times better than the bound in \cite{Beck13}
for either $P_k = L_k $ or $P_k = L$.
For example, when $P_k = L, \forall k$,  the bound in  \cite{Beck13} becomes $\mathcal{O}(\frac{KL}{r})$,
while our bound is  $\mathcal{O}(\frac{L}{r})$, which matches GD (listed in Table 1 below).
Another advantage of the new bound $(\sum_k L_k )^2$ is that it does not increase if we add an artificial block $x_{K+1}$ and perform CGD
for function $\tilde{g}(x, x_{k+1}) = g(x)$; in contrast, the existing bound $K L^2$ will increase to $(K+1) L^2$, even though the algorithm does not change at all.

We have demonstrated that our bound can match GD in some cases, but can possibly be $K$ times worse than $GD$.
An interesting question is: for general convex problems can we obtain an $\mathcal{O}(\frac{L}{r})$ bound for cyclic BCGD, matching the bound of GD? Removing the $K$-factor in \eqref{Beck} will lead to an $\mathcal{O}(\frac{L}{r})$ bound for conservative stepsize $P_k=L$ no matter how large $L_k$ and $L$ are.
Unfortunately, our approach which is based on estimating the norm of $H(\xi)$ probably cannot lead to removal of this $K$-factor as
 it seems unlikely to prove a bound $\| H(\xi)\| \leq L$.
Our bound $\| H(\xi) \| \leq \min \{K L^2, (\sum_k L_k)^2 \}$ is tight in some cases, but not always tight.
 A better bound of $\| H(\xi )\|$ (and thus a better iteration complexity bound) may be obtained by answering the following question (below $A_k$ represents the Hessian matrix corresponding to $\xi_k$, $k=1,\dots, K$, and $B(i,:)$ denotes the $i$'th row of a matrix $B$):
\begin{equation*}\label{Question of bound H}
\begin{split}
& \text{ Given } K\times K \text{ matrices } A_1, \dots, A_K \succeq 0 \text{ satisfying diag}(A_i)= ( L_1, \dots, L_K), \forall i \text{ and }  \\
& \| A_k \| \leq L, \text{ what is the best upper bound of } \| H\|, \text{ where } H = [A_1(1,:), \dots, A_i(i,:), \dots,  A_K(K,:)] ?
\end{split}
\end{equation*}
The answer to this question seems to be rather complicated and is left as future work.
%

We conjecture that an $\mathcal{O}(\frac{L}{r})$ bound for cyclic BCGD cannot be achieved for general convex problems. That being said, we point out that the iteration complexity of cyclic BCGD may depend on other intrinsic parameters of the problem such as $\{L_k\}_k$ and, possibly, third order derivatives of $g$. Thus the question of finding the best iteration complexity bound of the form
  $\mathcal{O}( h(K) \frac{L}{r} )$, where $h(K)$ is a function of $K$, may not be the right question to ask for BCD type algorithms.




{
\begin{remark}
Our complexity results apply directly to a popular variant of the BCD --  the permuted BCD, in which the blocks are randomly sampled {\it without replacement}. This is because our analysis only depends on the behavior of the algorithm during one iteration in which all blocks are updated once. The order of the block update across different iterations is irrelevant in the analysis. However, we mention that the key research question for permuted BCD is not whether it converges, but rather whether performing the random permutation delivers improved complexity bounds. To our knowledge this is still an open question that worth further investigation.
\end{remark}  }

\vspace{-0.2cm}
\section{Conclusion}
\vspace{-0.2cm}
In this paper, we provide new analysis and improved complexity bounds for cyclic BCD-type methods.
Our results also apply to BCD methods with random permutation (random sampling without replacement).
For minimizing the sum of a convex quadratic function and separable non-smooth functions, we
show that cyclic BCGD with small stepsize $1/L$ has an iteration complexity bound $\mathcal{O}(\log^2(2K) \frac{L}{r})$, which is independent of $K$
and matches the bound of GD  (except for a mild $O(\log^2(2K))$ factor).
For general smooth convex problems, we consider cyclic CGD and  prove a meta iteration complexity bound that is proportional to the spectral norm of a ``moving-iterate Hessian''.
This meta bound leads to an iteration complexity bound that is sometimes $K$-times better than existing bound of \cite{Beck13},
and matches the bound of GD for the quadratic case when using small stepsize $1/L$.
The derivation of our meta iteration complexity bound is rather tight, but still  this bound seems to be worse than the bound of GD in certain scenarios.
It remains an interesting open question whether this gap in the general convex case is artificial.


\newpage

{\centering\large{\bf Appendix}}

\subsection{Proof of The Tightness Result}\label{sec:tightness}
{\bf Claim}.
{\it Consider problem \eqref{eq:example:bcd} with $A=[A_1,\cdots, A_K]$ given by \eqref{eq:A}. Let
\begin{align}
x^{(0)} = \left[ 1, \frac{1}{8}, \frac{3}{4}, 1, 1, \cdots, 1, 1 \right]^T\in\mathbb{R}^K. \label{eq:x0}
\end{align}
Then after running BCD/BCPG for one iteration, the optimality gap $\Delta^{(r)}$ is bounded below by{
\begin{align}\label{eq:lower:bound}
{\Delta^{(1)}} \ge {L \|x^{(0)}-x^*\|^2}{}\frac{9(K-3)}{4(K-1)}, \; \forall~K\ge 3.
\end{align}}}

\begin{proof}
Assume that we start with a vector $x^{(0)}=\left[x^{(0)}_1, x^{(0)}_2,\cdots, x^{(0)}_K\right]^T$, let us generate the iterate $x^{(1)}$ by using cyclic BCD (assuming that coordinates are picked in the order of $1,2,\cdots, K$).
The first variable $x_1$ is updated by{
\begin{align}
x_1^{(1)}&=\arg\min_{x_1}\quad \left\|A_1 x_1+\sum_{j\ne 1}A_j x^{(0)}_j\right\|^2\nonumber\\
&=\arg\min_{x_1}\quad \left(x_1+x^{(0)}_2\right)^2+\left(x_1+x^{(0)}_2+x^{(0)}_3\right)^2.
\end{align}}
Here the second equality is true by the property of the matrix $A$. Clearly we have
\begin{align}
x^{(1)}_1=-\frac{1}{2}\left(2\times x^{(0)}_2+x^{(0)}_3\right). \label{eq:x_1}
\end{align}
The second variable $x_2$ is updated by
\begin{align}
x_2^{(1)}&=\arg\min\quad \left\|A_2 x_2+A_1 x^{(1)}_1+\sum_{j\ne 1,2}A_j x^{(0)}_j\right\|^2\nonumber\\
&=\arg\min\quad \left(x_2+x^{(1)}_1\right)^2+\left(x_2+x^{(1)}_1+x^{(0)}_3\right)^2+\left(x_2+x^{(0)}_4+x^{(0)}_3\right)^2. \nonumber
\end{align}
Clearly we have
\begin{align}
x^{(1)}_2=-\frac{1}{3}\left(2\times x^{(1)}_1+2\times x^{(0)}_3+x^{(0)}_4\right). \label{eq:x_2}
\end{align}
Similarly, the third variable $x_3$ is updated by
\begin{align}
x_3^{(1)}&=\arg\min\quad \left\|A_3 x_3+A_1x^{(1)}_1+A_2 x^{(1)}_2+\sum_{j\ne 1,2,3}A_j x^{(0)}_j\right\|^2\nonumber\\
&=\arg\min\quad \left(x_3+x^{(1)}_1+x^{(1)}_2\right)^2+\left(x_3+x^{(1)}_2+x^{(0)}_4\right)^2+\left(x_3+x^{(0)}_4+x^{(0)}_5\right)^2. \nonumber
\end{align}
The solution is given by
\begin{align}
x^{(1)}_3=-\frac{1}{3}\left(x^{(1)}_1+2\times x^{(1)}_2+2\times x^{(0)}_4+x^{(0)}_5\right). \label{eq:x_3}
\end{align}
In general, for all $k\in[3, K-2]$, we have
\begin{align}
x^{(1)}_k=-\frac{1}{3}\left(x^{(1)}_{k-2}+2\times x^{(1)}_{k-1}+2\times x^{(0)}_{k+1}+x^{(0)}_{k+2}\right). \label{eq:x_k}
\end{align}
Also we have that $x_{K-1}$ is updated by
\begin{align}
x_{K-1}^{(1)}&=\arg\min\quad \left\|A_{K-1} x_{K-1}+A_K x^{(0)}_K+\sum_{j\ne K-1,K}A_j x^{(1)}_j\right\|^2\nonumber\\
&=\arg\min\quad \left(x_{K-1}+x^{(1)}_{K-2}+x^{(1)}_{K-3}\right)^2+\left(x_{K-1}+x^{(1)}_{K-2}+x^{(0)}_K\right)^2+\left(x_{K-1}+x^{(0)}_{K}\right)^2. \nonumber
\end{align}
Therefore we have
\begin{align}
x^{(1)}_{K-1}=-\frac{1}{3}\left(2\times x^{(1)}_{K-2}+2\times x^{(0)}_K+x^{(1)}_{K-3}\right). \label{eq:x_K-1}
\end{align}
Similarly, we have
\begin{align}
x^{(1)}_{K}=-\frac{1}{2}\left(2\times x^{(1)}_{K-1}+x^{(1)}_{K-2}\right). \label{eq:x_K}
\end{align}
Now use the following initial solution:
\begin{align}
x^{(0)} = \left[ 1, \frac{1}{8}, \frac{3}{4}, 1, 1, \cdots, 1, 1 \right]^T\in\mathbb{R}^K. \label{eq:x0}
\end{align}
Plugin the above iteration we will have
\begin{align}
x^{(1)}_1= -\frac{1}{2}, \quad x^{(1)}_2= -\frac{1}{2}, \quad \cdots, \quad x^{(1)}_{K-2}= -\frac{1}{2}, \quad x^{(1)}_{K-1}= -\frac{1}{6}, \quad x^{(1)}_{K}= \frac{5}{12}.
\end{align}
This means that after the first iteration, the objective becomes
\begin{align*}
g(x^{(1)})& = 1+\frac{9}{4}(K-3)+\frac{1}{16} +\frac{1}{16}\ge \frac{9}{4}(K-3).
\end{align*}
We can also verify that
$$\|x^{(0)}-x^*\|^2 = \sum_{k=1}^{K}(x^{(0)}_k-0)^2=K-2+(1/8)^2+(3/4)^2.$$
Combining the above result, we see that when applying the BCD/BCPG to the problem \eqref{eq:example}, the complexity bound is {\it at least}
\begin{align}\label{eq:lower:bound}
\frac{\Delta^{(1)}}{\|x^{(0)}-x^*\|^2 } \ge \frac{\frac{9}{4}(K-3)}{(K-2+(1/8)^2+(3/4)^2)}\ge \frac{9(K-3)}{4(K-1)}, \; \forall~K\ge 3
\end{align}
This completes the proof of the claim. \QED
\end{proof}

\newpage

\bibliographystyle{ieee}
\bibliography{biblio,ref}

\end{document}